\documentclass[12pt,reqno,letter]{article}

\usepackage[margin=1in]{geometry}

\usepackage{listings}
 \usepackage[usenames,dvipsnames]{color}
 \usepackage[colorlinks=true, pdfstartview=FitV, linkcolor=blue, citecolor=blue, urlcolor=blue]{hyperref}

\usepackage[nottoc,numbib]{tocbibind}
\usepackage{helvet}

\usepackage{amsmath,amsthm,amsfonts,amssymb,amscd}

\usepackage{mathtools} 
\usepackage{mathabx}
\usepackage{float}
\usepackage{graphicx}
\usepackage{makeidx}
\usepackage{enumerate}
\usepackage{mathtools}
\usepackage{xfrac}

\usepackage[mathscr]{euscript}
\usepackage{stmaryrd}
\usepackage{microtype}
\usepackage{booktabs}
\usepackage{cleveref}
\usepackage{bookmark}
\usepackage{mathrsfs}

\usepackage{emptypage}
\usepackage{tikz}

\DeclareMathAlphabet{\mathpzc}{OT1}{pzc}{m}{it}

\theoremstyle{plain}
\newtheorem{theorem}{\scshape Theorem}[section]

\newtheorem{lemma}[theorem]{\scshape Lemma}

\theoremstyle{definition}
\newtheorem{definition}[theorem]{\scshape Definition}

\newtheorem{remark}[theorem]{\scshape Remark}

\def\<{\langle}
\def\>{\rangle}

\def\H{{:\!\mathcal H\!:}}
\def\N{\mathbb{N}}
\def\P{\mathbb{P}}
\def\R{\mathbb{R}}

\def\Z{\mathbb{Z}}

\def\1{\mathbf{1}}

\def\inte#1{
\displaystyle\mathop{#1\kern0pt}^\circ
}

\newcommand{\beq}{\begin{equation}}
\newcommand{\eeq}{\end{equation}}
\newcommand{\ben}{\begin{eqnarray}}
\newcommand{\een}{\end{eqnarray}}
\newcommand{\beno}{\begin{eqnarray*}}
\newcommand{\eeno}{\end{eqnarray*}}

\let\b=\beta

\let\G= \Gamma

\def\dB{\dot{B}}
\def\dH{\dot{H}}

\def\cF{{\mathcal F}}

\def\virgp{\raise 2pt\hbox{,}}
\def\cdotpv{\raise 2pt\hbox{;}}

\def\C{\mathop{\mathbb C\kern 0pt}\nolimits}
\def\DD{\mathop{\mathbb D\kern 0pt}\nolimits}
\def\EE{\mathop{\mathbb E\kern 0pt}\nolimits}
\def\K{\mathop{\mathbb K\kern 0pt}\nolimits}
\def\N{\mathop{\mathbb  N\kern 0pt}\nolimits}
\def\Q{\mathop{\mathbb  Q\kern 0pt}\nolimits}
\def\R{{\mathop{\mathbb R\kern 0pt}\nolimits}}
\def\SS{\mathop{\mathbb  S\kern 0pt}\nolimits}
\def\St{\mathop{\mathbb  S\kern 0pt}\nolimits}
\def\Z{\mathop{\mathbb  Z\kern 0pt}\nolimits}
\def\ZZ{{\mathop{\mathbb  Z\kern 0pt}\nolimits}}
\def\H{{\mathop{{\mathbb  H\kern 0pt}}\nolimits}}
\def\PP{\mathop{\mathbb P\kern 0pt}\nolimits}
\def\TT{\mathop{\mathbb T\kern 0pt}\nolimits}

\newcommand{\andf}{\quad\hbox{and}\quad}

\def\dive{\mathop{\rm div}\nolimits}

\def\Supp{\mathop{\rm Supp}\nolimits\ }

\usepackage{geometry}

\begin{document}

\title{A Lower  Bound Estimate  of Life Span of Solutions to Stochastic 3D Navier-Stokes Equations with Convolution Type Noise}
\maketitle

\author{Siyu Liang \footnotemark[1]\footnotemark[2]\footnotemark[3]
        }
        \footnotetext[1]{Academy of Mathematics and Systems Science, Chinese Academy of Sciences, Beijing 100190, China}
        \footnotetext[2]{School of Mathamatical Sciences,  University of Chinese Academy of Sciences, Beijing 100049, China}
        \footnotetext[3]{Department of Mathematics,  University of Bielefeld, D-33615 Bielefeld, Germany}

\begin{abstract}
In this paper we investigate the stochastic 3D Navier-Stokes equations perturbed  by linear multiplicative Gaussian noise of convolution type by transformation to random PDEs.   We are not interested  in the regularity of the initial data.    We focus on obtaining bounds from below for the life span associated with regular initial data.
 The key point of the proof is the fixed point argument.
\end{abstract}


\section{Introduction}

Consider the following stochastic $3D$ Navier--Stokes equation 
 \begin{equation}\label{C3S1eq1}
\left\{
\begin{aligned}
& du+(u\cdot\nabla u-\Delta u )dt=\sum\limits_{i=1}^{n}(B_i(u)+\lambda_{i}u)d\beta_{i}(t)-\nabla p dt, \\
&\dive u=0, \\
&u\mid_{t=0}=u_{0},
\end{aligned}
\right .
\end{equation}
 on the whole space $\mathbb{R}^3$,
where $\beta_{i}(t)$, $i=1,...,n$ are one dimensional independent Brownian motions on a given probability space $(\Omega, 
\mathcal{F},\P)$, $\lambda_{i}$, $i=1,...,n$,  are non-zero constants and $B_i$ , $i=1,...,n$  are the convolution operators such that 
$$
B_i(u)(\xi)=\int_{\mathbb{R}^3}h_i(\xi-\bar\xi)u(\bar\xi)d\bar\xi
=(h_i*u)(\xi),\ \xi\in\mathbb{R}^3 ,$$
where $h_i\in L^1(\mathbb{R}^3),\ i=1,2,...,n,$  and $\Delta$ is the (weak) Laplacian on $(L^2(\mathbb{R}^3))^3$. The vorticity form of this system has been investigated in
\cite{BARBU20175395} by Barbu and R\"{o}ckner,   where the authors
prove the existence and uniqueness in $(L^p(\mathbb{R}^3))^3$,
$\frac32<p<2$, of a    global mild solution to  random vorticity
equations associated to stochastic $3D$ Navier--Stokes equations
for sufficiently small initial vorticity. In their paper the smallness of the initial values  depend on the whole Brownian path, hence the solutions obtained  
are not adapted.   
In the paper \cite{Rockner2019} by 
R\"{o}ckner, R. C.  Zhu and X. C. Zhu,   the authors prove that  the solution satisfies the vorticity  equation
with the stochastic integration being understood in the sense of the integration of controlled rough paths. 
 In the further paper \cite{Munteanu} , the authors generalize this result to gradient-type noise in 2 or 3 dimensions by a  different type of transformation which is adapted to their   noise.   Then they also obtain the existence of a solution adapted to the Brownian filtration up to some stopping time.

We do not assume the initial values are small. 
Instead we only assume initial data are smooth. 
In other words,    the initial data are in  Sobolev spaces $H^N$ for any $N\in \N$.  
We focus on the life span of the  $\dH^{\frac{1}{2}+\gamma}$-mild  solution  on the fixed path (see Definition
\ref{mildsolution} and Theorem \ref{C3S3thm1} for the definitions of the mild solution and life span).

\textbf{ Main result of the paper (Theorem  \ref{C3S3thm1}):  }
we obtain  a lower bound estimate \eqref{C3S3eq13} of the life span.

For the deterministic classical $3D$ Navier--Stokes equations,  by Fujita--Kato's fixed point procedure,  global wellposedness
results have  been obtained in certain scaling invariant spaces (one also calls them critical spaces). 
For general initial values,  the solution may  blow up (in the sense of strong solutions) after some time.
For some spaces which are above the critical spaces,  a
fixed point procedure (Picard's contraction principle) can also be applied to study the blow up time of the solution.   We refer to  Chapter 15 in the  book \cite{lemarie2002recent} for such results for the Sobolev spaces 
$\dH^s$ and $L^p$, for $s \geq\frac{1}{2}$ and $p\geq 3$.
Also in Poulon's  paper \cite{Poulon2014},  the author introduces the notion of the minimal blow up Navier--Stokes solutions.  The authors also  show that the set of such solutions is not only nonempty but also compact in a certain sense.
Based on this,  a lower bound estimate of the maximal time up to which the solution 
remains regular for  initial values in the space 
$\dH^s$,  $\frac{1}{2}<s<\frac{3}{2}$, is obtained in  Proposition 1.1 of a later paper \cite{chemin2019}
by J. -Y. Chemin and  I. Gallagher.
Note that if we find some time $T$ and a solution in the space $L^{\infty}([0,T],  \dH^s)$,
for $\frac{1}{2}<s<\frac{3}{2}$,  then by Sobolev's embedding theorem we know that the solution up to time $T$ satisfies Serrin's condition,  i.e.
$$u\in L^q([0,T]; L^p(\mathbb{R}^{3})) \text{  for  some   } p,q  \text{ which satisfy   } \frac{2}{q}+\frac{3}{p}\leq 1.
$$
In other words,   it is a strong solution up to time $T$.  
For deterministic anisotropic $3D$ Navier--Stokes equations,  
in \cite{LIANG2020},
we also study the maximal time up to which the solutions remain regular.

The idea of the proof   of    Theorem  \ref{C3S3thm1}   is that we first apply the transformation in  \cite{BARBU20175395} to  transform the equation to a random PDE and write
the solution in the form of mild solutions.  
Then we apply Littlewood-Paley theory,  thanks to the commutativity  of the convolution (this is the reason why we have to limit our noise to  convolution type multiplicative noise),  our convolution operator $B_i$ and  transformation operator $\Gamma$  operate on any Besov space.
Afterwards,  due to the   contraction property of the semigroup $e^{t\Delta}$, we  obtain the estimates necessary  for the
fixed point argument.    This method  has also been  used to calculate a lower bound of the time $T$ up to which the regular solution exists  for deterministic $3D$ Navier--Stokes equations.

In this paper we use $C$ to denote the constant which can be different from line to line.
And we use the notation `$A\lesssim B$' to mean that
there is some constant $C$ such taht $A\leq CB$.\\

\section{Function Spaces on $\R^3$}\label{sec2}
Denote by $\mathcal{S}(\mathbb{R}^3)$ the Schwartz space and 
$\mathcal{S}'(\mathbb{R}^3)$ its dual space.\\
On $\mathbb{R}^{3}$, we recall the  \textbf{non-homogeneous Sobolev spaces}:
\beq\label{defhs}
H^{s}(\mathbb{R}^{3}):=\Bigl\{u\in \mathcal{S}'(\mathbb{R}^{3});\parallel u \parallel_{H^{s}(\mathbb{R}^{3})}^{2}:=\int_{\mathbb{R}^{2}}(1+\mid \xi \mid^{2})^s\mid \hat{u}(\xi)\mid^{2}d\xi<\infty\,\Bigr\},   
\eeq
where $s\in\R$, 
 and 
 $$\hat{u}(\xi)=\mathcal{F}u(\xi):=\int_{\R^3}u(x)e^{- ix\cdot \xi }dx,
   $$
  denotes the Fourier transform of $u$ on $\R^3$.\\
Then    $H^{s}(\mathbb{R}^{3})$ is a Hilbert space with $H^{-s}(\mathbb{R}^{3})$ as its dual space.

\noindent
On $\mathbb{R}^{3}$, we recall the \textbf{homogeneous Sobolev spaces}:
$$\dH^{s}(\mathbb{R}^{3}):=\Bigl\{u\in \mathcal{S}'(\mathbb{R}^{3}),
\hat{u}\in L^1_{loc}(\mathbb{R}^{3}) ;\parallel u \parallel_{\dH^{s}(\mathbb{R}^{3})}^{2}:=\int_{\mathbb{R}^{3}}\mid \xi \mid^{2s}\mid \hat{u}(\xi)\mid^{2}d\xi<\infty\,\Bigr\}, $$
where $\hat{u}$ denotes the Fourier transform of $u$.

\section{Transform to Random PDEs}
We use the same transform as  in \cite{BARBU20175395}, but instead of vorticity equation, we apply it to the original equation.
For $t\geq 0$, we consider the transformation
$$
u(t)=\Gamma(t)y(t), 
$$ where
$\Gamma(t):(L^2(\mathbb{R}^3))^3\to(L^2(\mathbb{R}^3))^3$ is the linear
continuous operator defined by the equations
\begin{equation*}
d\Gamma(t) = \sum^n_{i=1} (B_i+\lambda_i I)\Gamma(t)d\b_i(t),\ t\ge0,\
\end{equation*}
 and $\Gamma(0)=I$.
 In other word, 
   $\Gamma(t)$ is defined in
the sense that, for every $z_0\in (L^2(\mathbb{R}^3))^3$, the continuous
$(\mathcal{F}_t)$--adapted $(L^2(\mathbb{R}^3))^3$--valued process
$z(t):=\G(t)z_0$, $t\ge0$, solves the following SDE on
$(L^2(\mathbb{R}^3))^3$,
$$dz(t)=\sum^n_{i=1}\tilde B_iz(t)d\beta_i(t),\ \ z(0)=z_0.$$
Similar as in \cite{BARBU20175395},
we also set
\begin{equation*}
\tilde B_i= B_i+\lambda_i I,\ i=1,...,n,
\end{equation*} where $I$ is the
identity operator.
Then  we can also write $\Gamma$ as the exponential form:
$$\Gamma(t)=\prod^N_{i=1}\exp\left(\b_i(t)\tilde B_i-\frac{t}{2}\ \tilde
B^2_i\right),\ t\geq 0.$$
Moreover,    we immediately have all of $\Gamma(t)$,    $\Gamma^{-1}(t)$ and $\tilde B_i$ commutate with (weak) derivatives. Note that 
$$\Gamma^{-1}(t)=\prod^N_{i=1}\exp\left(-\b_i(t)\tilde B_i+\frac{t}{2}\ \tilde
B^2_i\right),\ t\geq 0.$$
Moreover, 
$\tilde B_i \tilde B_j=\tilde B_j \tilde B_i$.

\begin{remark}
\
{\sl
It is obvious that the operator $B_i$, $\Gamma(t)$ and $\Gamma^{-1}(t)$ can be defined (as a continuous operator) in any
$(L^p(\mathbb{R}^3))^3$ for any $p\geq 1$ since the convolution with an $L^1$ function makes sense in any  $L^p$ space with the $L^1$ norm as the uniform  bound of the operator (Young's inequality).
Moreover, we have the following lemma 
as is proved in \cite{BARBU20175395}.
\begin {lemma}\label{S2lem1}
{\sl We have
\begin{equation*}
\|\Gamma(t)z\|_{L^q}+\|\Gamma^{-1}(t)z\|_{L^q} \leq  C_t\|z\|_{L^q},\ t\in [0,\infty),\ \forall
z\in  L^q(\mathbb{R}^3),\ \forall q\in[1,\infty),
\end{equation*} and
\begin{equation*}
\|\nabla(\G(t)z) \|_{L^q} \leq \|\G(t)\|_{L(L^q,L^q)}\|\nabla z\|_{L^q},\mbox{ for all  $z$  which satisfies  $z, \nabla z \in L^q(\mathbb{R}^3)$}. 
\end{equation*}}
\end{lemma}
\begin{proof}
See Lemma 2.1 of \cite{BARBU20175395}.
\end{proof}\\
The above lemma also holds when $q=\infty$.  Since the proof is a direct result of 
Young's inequality
$$\|h_i*u\|_{L^q}\lesssim  \|h\|_{L^1} \|u\|_{L^q},
$$
which also holds as $q=\infty$, and the second inequality is due to the fact that $\Gamma$ commutates with derivatives.\\
 From Lemma \ref{S2lem1} we can view $\Gamma$ as a linear continuous operator from $(L^p(\mathbb{R}^3))^3$ to $(L^p(\mathbb{R}^3))^3$
for $1\leq p\leq \infty$.  Moreover, there is a common upper bound of 
$\|\Gamma(t)\|_{L(L^p,L^p)}$ which does not depend on $p$, but only
depends on $h_i$,$\lambda_i$, $t$ and, of course, the path $\omega$. }
\end{remark}

The next lemma tells us for any $s\in \R$, and  $h\in L^1$,    the convolution with $h$ is a continuous operator map from any Sobolev space $H^s$ to itself.
Therefore,  we can also extend the definition of  operator 
$B_i$, $\Gamma(t)$ and $\Gamma^{-1}(t)$ to
 continuous operators from any Sobolev space $H^s$ to itself.

\begin{lemma}\label{gammasobolev}
{\sl  For any $h\in L^1$ and $s\in\R$, the convolution with $h$ is a continuous operator mapping from any Sobolev space $H^s$ to itself.
}
\begin{proof}
It suffices to prove that for any $s>0$,   and $a\in H^s$,
we have $a\ast h \in H^s$.

\begin{equation*}
\begin{split}
\|a\ast h\|^2_{ H^s}&=  \int_{\R^3} |\mathcal{F}(a\ast h)|^2(\xi)(1+|\xi|^2)^sd\xi\\
&=\int_{\R^3} |\hat{a}(\xi)|^2 |\hat{h}(\xi)|^2(1+|\xi|^2)^sd\xi\\
&\leq \|a\|^2_{ H^s}\|  \hat{h}\|^2_{ L^{\infty}}\\
&\leq \|a\|^2_{ H^s}\|  h\|^2_{ L^{1}}.
\end{split}
\end{equation*}
Therefore,  we have proved the result.

\end{proof}

\end{lemma}

Thus we transform \eqref{C3S1eq1} to the following random PDEs of $y$:
\begin{equation}\label{C3S2eq1}
\begin{split}
 d y+
\Gamma^{-1}(\Gamma  y \cdot\nabla \Gamma y) dt-\Delta y dt&= - \Gamma^{-1} \nabla p dt,\\
y(0)&=u_{0}.
\end{split}
\end{equation} 
Note that 
$$
\nabla p=\nabla (-\Delta)^{-1}\dive\dive (\Gamma y\otimes \Gamma y)=
\sum\limits_{1\leq i,j\leq 3}\nabla (-\Delta)^{-1}\bigl(\partial_{i}
\partial_{j}(\Gamma y^i \Gamma y^j)\bigr).  $$
Let $Q$ be the following bilinear operator from $(L^2(\mathbb{R}^3))^3\times (L^2(\mathbb{R}^3))^3 $
to $(\mathcal{S}'(\mathbb{R}^3))^3$:
$$Q(x,y):=\dive (x\otimes y)+\sum\limits_{1\leq i,j\leq 3}\nabla (-\Delta)^{-1}(t)\bigl(\partial_{i}\partial_{j}(x^i y^j)\bigr)=Q(y,x).$$
Thus we can rewrite the equation in the following form of mild solution
\beq\label{mildsol}
y(t)=e^{t\Delta}u_0-\int_{0}^t e^{(t-s)\Delta}\Gamma^{-1}(\omega)Q\bigl(\Gamma(\omega) y(s), \Gamma(\omega) y(s)\bigr) ds.
\eeq

For simplicity from now on in this section we skip the dimension  notation $\mathbb{R}^3$ if it is 3 dimension and there is no confusion.
\begin{remark}
\
{\sl
\begin{enumerate}
\item 
$\Delta$ is an operator mapping from  tempered distribution space $\mathcal{S}'$ to $\mathcal{S}'$ .
\item We will show that $Q$, $\Gamma Q$  and $\Gamma^{-1} Q$  are well defined and continuous from $(L^p)^3\times (L^p)^3 $
to (at least) the Sobolev space  $(H^{-2})^3\subset (\mathcal{S}')^3$ for any $ 2\leq p< \infty$.
\\
\textbf{Case of $2<p<\infty$.}\\
It immediately follows from the $L^{\frac{p}{2}}$ boundedness of the Riesz transform.

\textbf{Case of $p=2$.}\\
Recall  that 
the inverse of the (minus)  Laplacian  $(-\Delta)^{-1}$  can be defined by Fourier multipliers:
for any $u$ which is in the range of the $\Delta$ ,  
$$(-\Delta)^{-1}u=\mathcal{F}^{-1}(|\xi|^{-2}\hat{u}(\xi)).
$$ 
Claim that 
for any $f\in L^1$,  $\partial_{i}\partial_{j}f 
$ is in the domain of $(-\Delta)^{-1}$ (i.e.  the range of the $\Delta$).\\
Indeed since  $f\in L^1$,  we have $\hat{f}\in L^{\infty}$, and $\xi_i\xi_j|\xi|^{-2}\hat{f}(\xi) \in L^{\infty}\subset\mathcal{S}'$.
Therefore, the term of Fourier inverse transform $\mathcal{F}^{-1}(\xi_i\xi_j|\xi|^{-2}\hat{f}(\xi))$ is at least meaningful in  $\mathcal{S}'$.

We have 
$$\Delta\mathcal{F}^{-1}(\xi_i\xi_j|\xi|^{-2}\hat{f}(\xi))=\partial_{i}\partial_{j}f
$$
and 
$$(-\Delta)^{-1}\bigl(\partial_{i}\partial_{j}f)=-\mathcal{F}^{-1}(\xi_i\xi_j|\xi|^{-2}\hat{f}(\xi)).
$$
Moreover,  since $\xi_i\xi_j|\xi|^{-2}\hat{f}(\xi) \in L^{\infty}$,
$$\|\mathcal{F}^{-1}(\xi_i\xi_j|\xi|^{-2}\hat{f}(\xi))\|_{H^{-2}}^2
=\int_{\R^3} |\xi_i\xi_j|\xi|^{-2}\hat{f}(\xi)|^2(1+|\xi|^2)^{-2}d\xi\leq \|f\|_{L^1}<\infty.
$$
Therefore,  $Q$ maps $(L^p)^3\times (L^p)^3 $
continuously to $(H^{-2})^3$.
Therefore, by Lemma \ref{gammasobolev} we know that $\Gamma Q$ and $\Gamma^{-1} Q$    also map $(L^p)^3\times (L^p)^3 $
continuously to $(H^{-2})^3$.

\item Therefore, by Sobolev embedding $Q$ and  $\Gamma Q$ can be well defined as  continous maps from 
 Sobolev space $(\dH^s)^3\times (\dH^s)^3$ for $0<s<\frac32$ to $(H^{-2})^3$ .
 
 \item From our definition of $Q(x,y)$, we immediately have  $\dive Q(x,y)=0$

\item
\label{remark5}
 Let $X$ be $L^p$, $p\geq 2$ or any Sobolev space $\dH^s$ for $0<s<\frac32$.  We have shown that
$Q$ is a continuous map from   $X^3\times X^3 $ to $(H^{-2})^3$,
thus if  $u$ and $v$ are  $\mathscr{L}([0,t])/\mathscr{B} (X^3)$-measurable,
$Q(u,v)$ is $\mathscr{L}([0,t])/\mathscr{B} ((H^{-2})^3)$-measurable,
where $\mathscr{L}([0,t])$ is the $\sigma$-algebra of the Lebesgue measurable sets in the interval $[0,t]$.

Since $X$   is dense in $(H^{-2})^3$,  and both $X$ and $(H^{-2})^3$ are Banach spaces,
we have $X\in \mathscr{B} ((H^{-2})^3)$ and 
$\mathscr{B} (X)=\mathscr{B} ((H^{-2})^3)\cap X$.   
\\
Therefore,  
$Q(u,v)1_{Q(u,v)\in X}$  is also $\mathscr{L}([0,t])/\mathscr{B} (X)$-measurable.

\item
Due to \ref{remark5},   and the fact that $\Gamma$ maps from any  Sobolev or Lebesgue space to itself
continuously,  (which we would prove later in the Remark \ref{C3S3lem4},    )
the integral in \eqref{mildsol} is meaningful in any Sobolev space if we can show the integration of their corresponding Sobolev norms are finite on the interval $[0,T]$ .

\end{enumerate}}

\end{remark}

\begin{lemma}
{\sl Assume that $\lambda_i, h_i$ satisfy 
 \begin{equation}\label{C3S2eq2}
|\lambda_i|>(\sqrt{12}+3)\|h_i\|_{L^1},\ \ \forall i=1,2,...,N.
\end{equation}
Let 
\begin{equation}\label{etatdef}
 \eta(t)=  \|\G(t) \|_{L(L^2,L^2)}^2
 \|\G^{-1}(t)\|_{L(L^2,L^2)} ,\ t\ge0,
\end{equation}
 where  for $q\in (1,\infty)$,
$\|\cdot\|_{L(L^q,L^q)}$ is the norm of the space $L(L^q,L^q)$ of
linear continuous operators on $L^q$.
Then we have
$$\sup_{t\geq0}\eta(t)<\infty, \mathbb{P}\mbox{-a.e.}$$}
\end{lemma}
\begin{proof}
The proof is the same to Remark 1.2 of \cite{BARBU20175395} if we 
note that the following still holds by calculation directly,
$$\eta(t)\le\prod^N_{i=1}\exp\bigl(3|\b_i(t)|
(\| h_i \|_{L^1}+|\lambda_i|)-t\alpha_i\bigr),\
t\in[0,\infty),$$ where $\alpha_i:=\frac{1}{2}
\lambda^2_i-\frac{3}{2}(\|h_i\|^2_{L^1}+2|\lambda_i|\,\|h_i\|_{L^1})$ 
is strictly positive followed by \eqref{C3S2eq2}.
\end{proof}

\section{The Main Theorem}

Define the function space $\mathcal{Z}_{T}^{\gamma}$ to be functions from $[0,T]\times \mathbb{R}^3$ to $\mathbb{R}^d$ ($d\geq 1$) with the corresponding norm
$$\|u\|_{\mathcal{Z}_{T}^{\gamma}}^2:=\|u\|_{L^\infty([0,T];\dH^{\frac{1}{2}+\gamma})}^2+\int_{0}^T\|\nabla u(t)\|_{\dH^{\frac{1}{2}+\gamma}}^2dt
$$
finite.\\
The completeness of $\mathcal{Z}_{T}^{\gamma}$  when
$0<\gamma<1$
is proved in the Appendix \ref{A}.

\begin{definition}[$\dH^{\frac{1}{2}+\gamma}$-mild solution on path $\omega$]\label{mildsolution}
{\sl   Fix $0<\gamma <1$ and the path $\omega$.  
Let $y(t,x)$ be a function from 
$[0,T]\times \R^3$
to $\R$.   
  We say that $y$ is an \textbf{$\dH^{\frac{1}{2}+\gamma}$-mild solution
   on path $\omega$}  to 
  \eqref{C3S2eq1}  in the interval $[0, T]$
if 
\begin{enumerate}
\item $y\in \mathcal{Z}_{T}^{\gamma}$;
\item \eqref{mildsol} holds for any $0\leq t\leq T$.
\end{enumerate}

}

\end{definition}

\begin{theorem}\label{C3S3thm1}
{\sl Fix $0<\gamma <1$.  
Assume that \eqref{C3S2eq2} holds.   
Given any (fixed)  vector field $u_0$ such that   $u_0\in \tilde{H}^N$ for any positive integer $N$,   
then for   $\P \mbox{-a.e.}$ path  $\omega$,       
 a positive time $T(u_{0},\omega)  $  exists such that a unique $\dH^{\frac{1}{2}+\gamma}$-mild  solution  on path $\omega$      to 
  \eqref{C3S2eq1} exists in the interval $[0,  T(u_{0},\omega)]  $.  
  Let $T^*(u_{0},\omega )$ be the supremum  of such $T(u_{0},\omega)  $.
  In this case we call $T^*(u_{0},\cdot )$ the \textbf{life span in $\dH^{\frac{1}{2}+\gamma}$ associated with the regular initial value $u_{0}$}. (Note that although $u_{0}$ is not random,  the life span depends on the path).
  Then we have $\P \mbox{-a.e.}$ path,  
\beq\label{C3S3eq13}
T^*(u_{0},\omega )\geq 
T_{*}(u_{0},\omega):=   c_{\gamma}\bigl(\sup\limits_{t\geq 0}\eta(t)\bigr)^{-1}\|u_0\|_{\dH^{\frac{1}{2}+\gamma}}^{-\frac{2}{\gamma}}, 
\eeq
 
and the strict positive number $c_{\gamma}$ depends only on $\gamma$.
}

\end{theorem}

We put the proof in  the Section \ref{C332}.
\begin{remark}
\
{\sl
\begin{enumerate}
\item In particular, in \eqref{C3S3eq13},  if  there is no noise, we do not need any transformation,
thus both $\Gamma$ and $\Gamma^{-1}$ are identity and $\eta_t=1$ for
any $t$, then we obtain the result which is consistent to the deterministic cases.
\item
Similar to \cite{BARBU20175395}, the solution we obtain is not adapted. 

\end{enumerate}}
\end{remark}
The proof of the main theorem relies heavily on the fixed point theorem. 
There are two key steps in the following proof: one is to show that $\Gamma$ and $\Gamma^{-1}$ could extend as operators in Besov spaces (see the next section for the definition of the Besov spaces), the proof of which relies on the commutative property of $\Gamma$ and Littlewood-Paley operators (we will introduce
the tool of Littlewood-Paley theory in the next section).      In other word, it relies on the commutative property of  convolutions since Littlewood-Paley operators are actually convolution operators.  This is the reason why we need the noise to be also convolution types. Another step is that just like what we do in deterministic equations, we write the solution in the form of mild solution and use the contraction property of the semigroup 
$e^{t\Delta}$ in order to obtain the estimates we need for the fixed point theorem.

\subsection{ Littlewood-Paley Theory}

Let us first recall the (homogeneous) Littlewood-Paley decomposition in the book 
\cite{Bahouri2011Fourier}. 
We will give a brief introduction of  Littlewood-Paley theory,  the details  of which could be found in he book 
\cite{Bahouri2011Fourier}. 
For $a\in \mathcal{S}'$,  as usual,  denote
$\cF a$ and $\widehat{a}$   the   Fourier transform of
the distribution $a$.
\begin{definition}[the Space $\mathcal{S}'_h$,  see Definition 1.26 of \cite{Bahouri2011Fourier} ]
{\sl  Let $\mathcal{S}'_h$ be the space of tempered distributions 
$u$ such that 
\beq\label{defsh}
\lim\limits_{\lambda\rightarrow \infty }\|\theta(\lambda D)u\|_{L^\infty}=0
\eeq
for any $\theta\in C_{c}^\infty$,
where $\theta(\lambda D)u$ is the Fourier multiplier defined as follows
$$\theta(\lambda D)u:=\mathcal{F}^{-1}(\theta(\lambda\cdot)\hat{u}).
$$}
\end{definition}
The next remark comes from Remark 1.27   of   \cite{Bahouri2011Fourier}
and the examples afterwards.
\begin{remark}
\
{\sl
\begin{enumerate}
\item Whether or not  a tempered distribution $u$ belongs to $\mathcal{S}'_h$ depends only on low frequencies.  $ u$  belongs to $\mathcal{S}'_h$  if and only if one can find some smooth compactly supported function $\theta$ such that  $\theta(0)\neq 0$ and \eqref{defsh}
holds.
\item Directly by the definition we immediately  know that the space $\mathcal{S}'_h$ contains all the tempered distributions 
whose Fourier transforms are locally integrable around $0$. In particular,
all the Sobolev spaces (homogeneous and non-homogeneous) are subsets
of $\mathcal{S}'_h$.
\end{enumerate}
}

\end{remark}

For $a\in \mathcal{S}'_h$, we set
\beq
\begin{split}
&\dot{\Delta}_k a=\cF^{-1}(\varphi(2^{-k}|\xi |)\widehat{a}),
\end{split} \label{C3S3eq1}\eeq
where $\varphi(\tau)$ is a 
smooth function value in $[0,1]$ such that 
 \beno
&&\Supp \varphi \subset \Bigl\{\tau \in \R;  \ \frac34 \leq
|\tau| \leq \frac83 \Bigr\}\andf \  \ \forall
 \tau>0\,,\ \sum_{k\in\Z}\varphi(2^{-k}\tau)=1.\\
 \eeno
 Then we have  
 $$\forall \xi\in \mathbb{R}^3 \backslash \{0\},   
 \frac{1}{2} \leq  \sum\limits_{j\in\mathbb{Z}}\varphi^2(2^{-j}\xi)\leq 1. $$
Moreover, we have the following equality for $a\in \mathcal{S}'_h$,
\beq\label{C3S3eq2}
a=\sum\limits_{j\in \mathbb{Z}}\dot{\Delta}_ja.
 \eeq
 For $s\in \mathbb{R}$ and $(p,r)\in [1,\infty]^2$, define the homogeneous Besov spaces $\dB_{p,r}^s$ which consists of those distributions in $\mathcal{S}'_h$ such that
 $$
 \|u\|_{\dB_{p,r}^s}:=\bigl(\sum_{j\in\Z}2^{rjs}\|\dot{\Delta}_j u\|_{L^p}^r\bigr)^{\frac{1}{r}}<\infty.
 $$
 From the definition we immediately know the norm of 
 $\dH^s$ coincides with $\dB_{2,2}^s$.
\begin{lemma}\label{C3S3lem3}
{\sl
For $0<\gamma<1$,
 $$\|u\otimes u\|_{\dB_{2,1}^{2\gamma-\frac{1}{2}}}\lesssim \|u\|_{\dH^{\frac{1}{2}+\gamma}}^2
$$}
\begin{proof}
See Corollary 2.55 of \cite{Bahouri2011Fourier}.

\end{proof}
\end{lemma}

 \begin{remark}\label{C3S3lem4}
 \
 
 {\sl
 \begin{enumerate}
 \item  By definition $\dot{\Delta}_j $ commutates with $\Gamma(t)$ and $\Gamma^{-1}(t)$. That is,
for  $u\in L^p$, we have
\beq\label{C3S3eq12}
\Gamma \dot{\Delta}_j u=\dot{\Delta}_j \Gamma u \text{  and }
\Gamma^{-1} \dot{\Delta}_j u=\dot{\Delta}_j \Gamma^{-1} u.
\eeq
\item  From \eqref{C3S3eq2}, we know for any $p$, $L^p$ is dense in $\mathcal{S}'_h$,
thus we can extend $\Gamma(t)$ and $\Gamma^{-1}(t)$ continuously and uniquely to an operator from  
$ \mathcal{S}'_h $ to $ \mathcal{S}'_h$.

\item We claim for $u\in   \dB_{p,r}^s$, 
\beq\label{C3S3eq11}
\Gamma u=\sum\limits_{j\in\mathbb{Z}} \dot{\Delta}_j \Gamma u=\sum\limits_{j\in\mathbb{Z}}\Gamma \dot{\Delta}_j u,
\eeq
and
\beq\label{C3S3eq10}
\|\Gamma(t)\|_{L(\dB_{p,r}^s, \dB_{p,r}^s)}\leq \|\Gamma(t)\|_{L(L^p,L^p)},
\eeq
where $L(X,X)$ denote the operator norm from $X$ to $X$.\\
Indeed, first we note that \eqref{C3S3eq12} still holds for $u\in \dB_{p,r}^s$,
since by the way of expansion, we have
$$\Gamma u:=    \sum\limits _{k\in\mathbb{Z}}\Gamma\dot{\Delta}_k u  ,
$$
and 
\begin{equation*}
\begin{split}
\Gamma \dot{\Delta}_j u&=\sum\limits _{k\in\mathbb{Z}}\Gamma \dot{\Delta}_k \dot{\Delta}_j u\\
&=\sum\limits _{k\in\mathbb{Z}}\Gamma \dot{\Delta}_j \dot{\Delta}_k u\\
&=\sum\limits _{k\in\mathbb{Z}} \dot{\Delta}_j\Gamma \dot{\Delta}_k u\\
&=\dot{\Delta}_j\Gamma u,
\end{split}
\end{equation*}
where the third inequality is due to \eqref{C3S3eq12} in $L^p$.
Thus we finish the proof of  \eqref{C3S3eq11}.
For the proof of \eqref{C3S3eq10}, immediately follows by the definition of Besov norms we have
\begin{equation*}
\begin{split}
\|\Gamma u\|_{\dB_{p,r}^s}&=\bigl(\sum_{j\in\Z}2^{rjs}\|\dot{\Delta}_j \Gamma u\|_{L^p}^r\bigr)^{\frac{1}{r}}\\
&=\bigl(\sum_{j\in\Z}2^{rjs}\| \Gamma \dot{\Delta}_j  u\|_{L^p}^r\bigr)^{\frac{1}{r}}\\
&\leq \bigl(\sum_{j\in\Z}2^{rjs}\| \Gamma\|_{L(L^p,L^p)}^r\|  \dot{\Delta}_j  u\|_{L^p}^r\bigr)^{\frac{1}{r}}\\
&\leq \| \Gamma\|_{L(L^p,L^p)} \| u\|_{\dB_{p,r}^s}.
\end{split}
\end{equation*}

  \end{enumerate}}
  
 \end{remark}

 Roughly speaking, the above preparation is to show how `good' the operators $\Gamma$  and $\Gamma^{-1}$ are. They could be expanded to any Besov (hence Sobolev) space and could  commutate with derivatives and Littlewood-Paley operators.
 After the preparation, we now show the following lemma, 
which is the crucial step of the proof of the main theorem.

\subsection{Proof of Main Theorem}\label{C332}

Define 
$$F(y)(t)=-\int_{0}^t e^{(t-s)\Delta}\Gamma^{-1}Q(\Gamma y, \Gamma y) ds.
$$

\begin{lemma}\label{C3S3lem2}
{\sl There exists some constant C, which depends on the path $\omega$, such that 
$$\|F(y)\|_{\mathcal{Z}_T^\gamma}\leq C T^{\frac{\gamma}{2}}\|y\|_{\mathcal{Z}_T^\gamma}^2.
$$}
\end{lemma}
\begin{proof}
The key point of the proof is based on the contraction property of the 
semigroup $e^{t\Delta}$. 
By definition
\begin{equation*}
\begin{split}
\|F(y)(t)\|_{\dH^{\frac{1}{2}+\gamma}}&\leq \int_{0}^t \|e^{(t-s)\Delta}\Gamma^{-1}Q(\Gamma y(s), \Gamma y(s))\|_{\dH^{\frac{1}{2}+\gamma}} ds
\\
&\lesssim \int_{0}^t \|e^{(t-s)\Delta}\Gamma^{-1} (\Gamma y(s)\otimes \Gamma y(s))\|_{\dH^{\frac{3}{2}+\gamma}} ds,
\end{split}
\end{equation*}
where the second inequality is due to the reason that $Q(\Gamma y,\Gamma y)$ is the sum of 
first order derivatives of $\Gamma y\otimes \Gamma y$.
Moreover, by Lemma 2.4 of \cite{Bahouri2011Fourier}, there exists some constant $c$,
\beq\label{C3S3eq3}
\begin{split}
&\|\dot{\Delta}_j [e^{(t-s)\Delta} \Gamma^{-1} (\Gamma y\otimes \Gamma y)]\|_{L^2} \\
&\lesssim e^{-c(t-s)2^{2j}} \| \dot{\Delta}_j[ \Gamma^{-1} (\Gamma y\otimes \Gamma y)]\|_{L^2}\\
&\lesssim e^{-c(t-s)2^{2j}}\|\Gamma^{-1}(s)\|_{L(L^2,L^2)}2^{-j(2\gamma-\frac{1}{2})}d_j \|\Gamma y\otimes \Gamma y\|_{\dB_{2,1}^{2\gamma-\frac{1}{2}}}\\
&\lesssim e^{-c(t-s)2^{2j}}\|\Gamma^{-1}(s)\|_{L(L^2,L^2)}2^{-j(2\gamma-\frac{1}{2})}d_j \|\Gamma y \|_{\dH^{\frac{1}{2}+\gamma}}^2\\
&\lesssim \eta_s e^{-c(t-s)2^{2j}}2^{-j(2\gamma-\frac{1}{2})}\| y(s) \|_{\dH^{\frac{1}{2}+\gamma}}^2 d_j  ,
\end{split}
\eeq
where $d_j$ is a sequence in $\ell^1$, the third inequality is due to Lemma \ref{C3S3lem3} and the last inequality is due to the reason that from Remark \ref{C3S3lem4}, 
we know 
$$\|\Gamma(t)\|_{L(\dH^{\frac{1}{2}+\gamma}, \dH^{\frac{1}{2}+\gamma})}\leq \|\Gamma(t)\|_{L(L^2,L^2)}.
$$
Thus we have
$$\|e^{(t-s)\Delta}\Gamma^{-1} (\Gamma y\otimes \Gamma y)\|_{\dH^{\frac{3}{2}+\gamma}}
\lesssim \eta_{s} \| y(s) \|_{\dH^{\frac{1}{2}+\gamma}}^2
\sup\limits_j[e^{-c(t-s)2^{2j}}2^{j(2-\gamma)}].$$
Since there exists some constant $C(\gamma)$,
such that 
$$e^{-c(t-s)2^{2j}}\leq C(\gamma)[(t-s)2^{2j}]^{-1+\frac{\gamma}{2}},
$$
we obtain
\beq\label{C3S3eq5}
\|e^{(t-s)\Delta}\Gamma^{-1} (\Gamma y\otimes \Gamma y)\|_{\dH^{\frac{3}{2}+\gamma}}\lesssim \eta_{s} \| y(s) \|_{\dH^{\frac{1}{2}+\gamma}}^2(t-s)^{-1+\frac{\gamma}{2}}.
\eeq
Therefore,
\beq \label{C3S3eq4}
\|F(y)(t)\|_{L^\infty([0,T];\dH^{\frac{1}{2}+\gamma})}\lesssim T^{\frac{\gamma}{2}}
\| y(s) \|_{L^\infty([0,T];\dH^{\frac{1}{2}+\gamma})}^2\sup\limits_{t\geq 0}\eta_t.
\eeq
On the other hand,
\beq \label{C3S3eq7}
\|\nabla F(y)(t)\|_{\dH^{\frac{1}{2}+\gamma}}
\leq \int_{0}^t \|e^{(t-s)\Delta}\Gamma^{-1} (\Gamma y(s)\otimes \Gamma y(s))\|_{\dH^{\frac{5}{2}+\gamma}} ds.
\eeq
Following the same way that we obtain \eqref{C3S3eq5},  by replacing $\gamma$ by $\gamma+1$, we obtain
\beq\label{C3S3eq6}
\|e^{(t-s)\Delta}\Gamma^{-1} (\Gamma y(s)\otimes \Gamma y(s))\|_{\dH^{\frac{5}{2}+\gamma}}\lesssim \eta_{s} \| y(s) \|_{\dH^{\frac{3}{2}+\gamma}}^2(t-s)^{\frac{\gamma-1}{2}}.
\eeq
Therefore, by combining \eqref{C3S3eq7} and \eqref{C3S3eq6},
we deduce
\begin{equation*}
\begin{split}
\int_{0}^T \|\nabla F(y)(t)\|_{\dH^{\frac{1}{2}+\gamma}}^2dt
&\leq \int_{0}^T \Bigl( \int_{0}^t \eta_{s} \| y \|_{\dH^{\frac{3}{2}+\gamma}}^2(t-s)^{\frac{\gamma-1}{2}}ds\Bigr)^2dt\\
&\lesssim \sup\limits_{t\geq 0}\eta_t^2
\int_{0}^T    \bigl(  \int_{0}^t  \| y(s) \|_{\dH^{\frac{3}{2}+\gamma}}^2 ds\  t^{\frac{\gamma-1}{2}}     \bigr)^2  dt\\
&\lesssim \sup\limits_{t\geq 0}\eta_t^2 \int_{0}^T t^{\gamma-1}dt
 \int_{0}^T \| y(t) \|_{\dH^{\frac{3}{2}+\gamma}}^2 dt\\
 &\lesssim \sup\limits_{t\geq 0}\eta_t^2\   T^\gamma 
 \int_{0}^T \| y(t) \|_{\dH^{\frac{3}{2}+\gamma}}^2 dt.
\end{split}
\end{equation*}
That is,
\beq \label{C3S3eq8}
\|\nabla F(y)(t)\|_{L^2([0,T];\dH^{\frac{1}{2}+\gamma})}\lesssim T^{\frac{\gamma}{2}}
\| \nabla y(s) \|_{L^2([0,T];\dH^{\frac{1}{2}+\gamma})}^2\sup\limits_{t\geq 0}\eta_t.
\eeq
The conclusion follows directly from \eqref{C3S3eq4} and \eqref{C3S3eq8}.

\end{proof}
\

\begin{lemma}\label{heatequation}
{\sl For any $T>0$ and $0<\gamma<1$,   
$$\| e^{t\Delta}u_0\|_{\mathcal{Z}_T^\gamma}\leq \|u_0\|_{\dH^{\frac{1}{2}+\gamma}}.
$$

}

\end{lemma}
\begin{proof}
The proof is trivial.  We write a simple proof here for completeness.\\
Set $v=e^{t\Delta}u_0$,
then $v$ is the solution of the following heat  equations
 \begin{equation}\label{C3heatequation}
\left\{
\begin{aligned}
\partial_t u&=\Delta u \\
u\mid_{t=0}&=u_{0}.
\end{aligned}
\right .
\end{equation}
Taking $\dH^{\frac{1}{2}+\gamma}$ inner product of both sides  
with $u$ immediately yields the result.

\end{proof}

The following  fixed point theorem  comes from Lemma 5.5 of \cite{Bahouri2011Fourier}:
\begin{lemma}\label{55BCD}
Let $E$ be a Banach space, $\mathcal{B}$ a continuous bilinear map from $E\times E$ to $E$,  and $\alpha$ a positive real number such that
$$\alpha <\frac{1}{4\|\mathcal{B}\|} \text{  with  }
\|\mathcal{B}\|:=\sup\limits_{\|u\|\leq 1, \|v\|\leq 1}\|\mathcal{B}(u,v)\|.
$$
For any $a$ in the ball $B(0,\alpha)$ (i.e., with center $0$ and radius 
$\alpha$) in $E$, a unique $x$ then exists in $B(0,2\alpha)$ such that
$$x=a+\mathcal{B}(x,x).$$
\end{lemma}

\noindent
\textbf{Proof of   Theorem \ref{C3S3thm1}      }\\
The result comes immediately by Lemma \ref{C3S3lem2},
Lemma \ref{heatequation} and Lemma \ref{55BCD}.

\qed

\appendix

\section{the Completeness of the Normed Space $\mathcal{Z}_T^\gamma$ }\label{A}
We prove the completeness of the space $\mathcal{Z}_T^\gamma$ when
$0<\gamma<1$.\\
Firstly,  recall that the norm of $\mathcal{Z}_T^\gamma$ is:
$$\|u\|_{\mathcal{Z}_{T}^{\gamma}}^2:=\|u\|_{L^\infty([0,T];\dH^{\frac{1}{2}+\gamma})}^2+\int_{0}^T\|\nabla u(t)\|_{\dH^{\frac{1}{2}+\gamma}}^2dt.
$$
Let $\{u_n\}_{n\geq 1}$ be a Cauchy sequence in $\mathcal{Z}_T^{\gamma}$,
which means $\{u_n\}_{n\geq 1}$  is a Cauchy sequence in $L^\infty([0,T];\dH^{\frac{1}{2}+\gamma})$ and $\{\nabla u_n\}_{n\geq 1}$ is a Cauchy sequence in $L^2([0,T];\dH^{\frac{1}{2}+\gamma})$.

Our aim is to find some $u\in \mathcal{Z}_T^\gamma $, such that 
$u_n$ converges to $u$ in the norm of  $\mathcal{Z}_T^{\gamma}$.
By Prop 1.34 of \cite{Bahouri2011Fourier}, we know that when 
$0<\gamma<1$,  the homogeneous Sobolev space $\dH^{\frac{1}{2}+\gamma}$
is a Hilbert space. 
Therefore,  both $L^\infty([0,T];\dH^{\frac{1}{2}+\gamma})$
and $L^2([0,T];\dH^{\frac{1}{2}+\gamma})$ are complete.
Therefore,  there exists $v_1$ and $v_2$,  such that
$$u_n\longrightarrow v_1 \text{  in    }   L^\infty([0,T];\dH^{\frac{1}{2}+\gamma}),$$
and 
$$\nabla u_n\longrightarrow v_2 \text{  in    }   L^2([0,T];\dH^{\frac{1}{2}+\gamma}).$$
Now note that if we can prove that  $\nabla v_1(t)=v_2(t)$ for $a.e. t\in[0,T]$,
(of course  the derivatives mean weak derivatives, which can be at least defined for 
those $v_1(t)$ which satisfy $v_1(t)\in \dH^{\frac{1}{2}+\gamma}$ and  the derivatives  are
at least in the space $\mathcal{S}'$),\\
then we immediately  have 
$$\nabla u_n\longrightarrow \nabla v_1 \text{  in    }   L^2([0,T];\dH^{\frac{1}{2}+\gamma}),$$
hence 
$ v_1\in  \mathcal{Z}_T^\gamma$ and 
$u_n$ converges to $v_1$ in the norm of  $\mathcal{Z}_T^{\gamma}$.\\
\textbf{To prove:  $\nabla v_1(t)=v_2(t)$ for $a.e. t\in[0,T]$:}\\
Since $L^p$ convergence implies convergence in probability, hence we can find a subsequence $n_k$,
such that 
$$u_{n_k}(t)\longrightarrow v_1(t) \text{  in    }   \dH^{\frac{1}{2}+\gamma} a.e.    , $$
and 
$$\nabla u_{n_k}(t)\longrightarrow v_2(t) \text{  in    }   \dH^{\frac{1}{2}+\gamma} a.e.  .$$
Therefore,   there exists a subset $A$ of the interval $[0,T]$ which has the full Lebesgue measure $T$, such that for any $t\in A$,
$$u_{n_k}(t)\longrightarrow v_1(t) \text{  in    }   \dH^{\frac{1}{2}+\gamma}  , $$
and 
$$\nabla u_{n_k}(t)\longrightarrow v_2(t) \text{  in    }   \dH^{\frac{1}{2}+\gamma} .$$
Then for any $\phi\in (C_c^{\infty}(\R^3))^3$, and $t\in A$, 
we have 
\begin{equation*}
\begin{split}
\langle  v_2(t),  \phi \rangle&=\lim\limits_{k\rightarrow \infty}
\langle \nabla u_{n_k}(t),  \phi \rangle\\
&=-\lim\limits_{k\rightarrow \infty}
\langle  u_{n_k}(t), \dive \phi \rangle\\
&=-\langle   v_1(t), \dive \phi \rangle\\
&=\langle   \nabla v_1(t),  \phi \rangle,
\end{split}
\end{equation*} 
where same as before,   derivatives mean weak derivatives,  $\langle \cdot,  \cdot \rangle$ is the
 duality bracket and the second and the  fourth equalities are due to the definition of the weak derivatives.
Therefore,  we have shown $\nabla v_1(t)=v_2(t)$ for $a.e. t\in[0,T]$, which finishes our proof.

\bigbreak \noindent {\bf Acknowledgments.}  
S. Liang is grateful for the financial support from Deutsche Forschungsgemeinschaft (DFG) through the program IRTG 2235.
The author thanks Prof.  Dr.  Michael R\"{o}ckner for pointing out the precise 
way to state Definition  \ref{mildsolution} and Theorem \ref{C3S3thm1}.
\setcounter{secnumdepth}{-1}
\bibliographystyle{alpha}

\end{document}